\documentclass[amsthm,reqno,12pt]{amsart}
\usepackage[utf8]{inputenc}
\usepackage[english]{babel}
\usepackage[symbol]{footmisc}
\usepackage{amssymb,amsmath,amsthm,amsfonts,xcolor,enumerate,hyperref,comment,longtable,cleveref}

\usepackage{times}
\usepackage{cite}
\usepackage{pdflscape}
\usepackage{ulem}
\usepackage[mathcal]{euscript}
\usepackage{tikz}
\usepackage{hyperref}
\usepackage{cancel} 
\usepackage{stmaryrd}
\usepackage[all]{xy}
\usepackage{pdfsync}
\usepackage{ytableau}
\usepackage{rotating}

\usetikzlibrary{arrows}

\def\Perm{\mathrm{Perm}}

\def\Der{\mathrm{Der}}

\def\wt{\mathop{\fam 0 wt}\nolimits}

\let\<\langle
\let\>\rangle

\theoremstyle{definition}
\newtheorem{definition}{Definition}[section]

\theoremstyle{plain}
\newtheorem{lemma}[definition]{Lemma}
\newtheorem{Proposition}[definition]{Proposition}
\newtheorem{proposition}[definition]{Proposition}
\newtheorem{corollary}[definition]{Corollary}
\newtheorem{theorem}[definition]{Theorem}
\newenvironment{Proof}[1][Proof.]{\begin{trivlist}
		\item[\hskip \labelsep {\bfseries #1}]}{\flushright
		$\Box$\end{trivlist}}

\begin{document}

\title{Identities in differential perm algebras}

\author{F. Mashurov and B. Sartayev}

\address{SICM, Southern University of Science and Technology, Shenzhen, 518055, China}

\address{SDU University, Kaskelen, Kazakhstan}

\email{f.mashurov@gmail.com}

\address{Narxoz University, Almaty, Kazakhstan}

\address{SDU University, Kaskelen, Kazakhstan}

\email{baurjai@gmail.com}

\subjclass[2020]{17A30, 17A50}
\keywords{Perm algebra, Lie algebra, Leibniz algebra, polynomial identities}


\begin{abstract}
Let $(P,\cdot,d)$ be a differential perm algebra over a field of characteristic zero, i.e.  an
associative algebra satisfying $(ab)c=(ba)c$ and equipped with a derivation $d$.
We study polynomial identities  in the algebras obtained by
the derived operations
\[
a\prec b=ab',\quad a\succ b=a'b,\quad
a\blacklozenge b=ab'+ba',\quad a\bullet b=a'b+ab',\quad
a\Diamond b=ab'-ba',\quad a\circ b=a'b-ab',
\]
where $a'=d(a)$.
We first prove that any nontrivial differential polynomial identity that does not belong to the right annihilator of the free differential perm algebra implies a differential identity of the form
$a_1'a_2'\cdots a_m'=0$ for some positive integer $m$.
We then obtain explicit generating sets and determine the dimensions of the multilinear homogeneous components of the subalgebras of the free differential perm algebra generated by $X$ with respect to the products $\blacklozenge$ and $\bullet$. Finally, we construct perm–Witt type Lie and Leibniz algebras arising naturally from differential perm algebras.
\end{abstract}

\maketitle

\section{Introduction}\label{sec:intro}

Let $\mathbb{K}$ be a field of characteristic zero.
An associative algebra $(P,\cdot)$ is called a perm algebra if it satisfies the
left-commutative identity
\begin{equation}\label{perm id}
x_1x_2x_3-x_2x_1x_3=0 \end{equation}

Perm algebras have a natural operadic interpretation related to dialgebras and relations between quadratic operads (including Koszul duality and Manin products)
\cite{Loday2001,DMS2027,GK2014,KS2026,Chapoton2001,Vallette2008}.  In particular, it is the duplicator of the operad of commutative associative algebras and is Koszul dual to the pre-Lie operad.   These relations make perm algebras a natural setting in which constructions arising from commutative associative algebras can be extended to the noncommutative case.

A \textit{differential perm algebra} is a perm algebra equipped with a derivation $d\colon P\to P$. 
The classical commutative case provides a natural motivation for
studying such algebras. Indeed, if $A$ is a commutative associative algebra with a
derivation $d$, then the product
\[
a\prec b:=ab', \qquad a'=d(a),
\]
defines a Novikov algebra structure on $A$. For a differential perm algebra, we have the two basic derived products
\begin{equation}
    a\prec b:=ab',
\qquad
a\succ b:=a'b, \qquad a'=d(a).
\end{equation} These products have been studied from several points of view. Kolesnikov and
Sartayev \cite{KolesnikovSartayev2022} proved that the product $a\prec b=ab'$ is left-symmetric and
characterized those left-symmetric algebras that can be embedded into differential
perm algebras by this product. The two derived operations also arise naturally in the
theory of Novikov dialgebras; their relation with differential perm algebras was
studied in \cite{SartayevDzhumadildaevMashurov2025}. 

On the other hand, Dotsenko, Ismailov, and Umirbaev \cite{DotsenkoIsmailovUmirbaev2023} studied polynomial identities on differential commutative associative algebras and showed that a nontrivial identity implies a strong differential consequence. This result provides one of the main motivations for considering the analogous problem for differential perm algebras.

The aim of the present paper is to study the algebras and polynomial identities associated with the basic derived products and their symmetrizations \[
a\blacklozenge b:=ab'+ba',
\qquad
a\bullet b:=a'b+ab',
\]
and antisymmetrizations 
\[
a\Diamond b:=ab'-ba',
\qquad
a\circ b:=a'b-ab'.
\]
Although these operations are defined by closely related formulas, they lead to different algebraic structures. Our purpose is to describe these structures and the subalgebras generated by the original set of generators. We first study polynomial identities in differential perm algebras. We prove that if a differential perm algebra satisfies a nontrivial differential polynomial identity $f=0$, such that $f$ does not belong to the right annihilator of the free differential perm algebra, then it also satisfies an identity of the form
\[
a_1'a_2'\cdots a_m'=0
\]
for some $m\geq 2$. This is a perm analogue of one of the crucial results appearing in \cite{DotsenkoIsmailovUmirbaev2023}. 

We next consider the subalgebras of the free differential perm algebra generated by $X$ with respect to the products $\blacklozenge$ and $\bullet$. For both products, we obtain explicit generating sets and determine the dimensions of their multilinear homogeneous components. The product $\blacklozenge$ is related to the symmetrized product arising from differential commutative associative algebras and satisfies Tortken-type identities, while $\bullet$ satisfies the corresponding dialgebraic identities. These results are related to the theory of Novikov--Jordan algebras \cite{Dzhumadil2002, Dzhumadil2005} and to the description of symmetric elements in free Novikov algebras \cite{DzhumadilIsmailov2026Null}.

We then turn to the skew derived operations. The product $a\Diamond b=ab'-ba'$
is the commutator of the left-symmetric product $a\prec b=ab'$ and hence defines a Lie bracket. In contrast,
$a\circ b=a'b-ab'$
naturally gives rise to Leibniz structures. We extend the latter construction to families of commuting derivations and to $\delta$-derivations, obtaining noncommutative transposed $\delta$-Leibniz--Poisson structures. Finally, we apply these constructions to a class of perm algebras obtained from polynomial algebras and construct corresponding Witt-type Lie and Leibniz algebras. In one variable, the construction reduces to the classical situation, whereas in several variables the noncommutativity of the underlying perm algebra is reflected in the resulting multiplication tables. These  constructions are also consistent with the broader role of perm algebras in the study of metabelian Lie algebras \cite{MashurovSartayev2024}; see also \cite{KaygorodovMashurov2024, ZhaoLiuPreprint2025,KunanbayevSartayev2026,LinQTPreprint2025, LinZhouBai2025,CuiHouOnline2026,PeiBaiGuoNi2020, Hou2024}.

\noindent\textbf{Organization of the paper.}
Section~2 recalls the free perm algebra and the free differential perm algebra, fixes the grading conventions used throughout, and proves the reduction theorem for differential polynomial identities. Section~3 studies symmetric elements and the $\blacklozenge$-product, including the basis and multilinear dimension of the generated subalgebra. Section~4 treats the corresponding weak-symmetric elements for the $\bullet$-product. Section~5 deals with Leibniz brackets and noncommutative transposed Leibniz--Poisson structures. Finally, Section~6 develops the Witt-type Lie and Leibniz constructions associated with differential perm algebras.

\section{Free differential perm algebras and polynomial identities}\label{sec:prelim}

In this section, we recall the basis of the free differential perm algebra and use it to study differential polynomial identities. Throughout the paper, we omit the multiplication $\cdot$ and write $ab$ instead of $a\cdot b$.

Let $X=\{x_1,x_2,\dots\}$ be a set and $\Perm\langle X\rangle$ be a free perm algebra generated by $X$. 

A linear basis of $\Perm\langle X\rangle$ is given by \[ x_{i_1}x_{i_2}\cdots x_{i_{n-1}}x_{i_n}, \qquad i_1\leq i_2\leq\cdots\leq i_{n-1}, \qquad n\geq1, \] where $i_n$ is arbitrary.


Let
$$X^{(\omega)}:=\{x^{(s)}\mid x\in X,\ s\ge 0\},\qquad x^{(0)}:=x.$$

We fix a total order on $X^{(\omega)}$,
$$x_i^{(r)}<x_j^{(s)}
 \text{ if }
r<s
\ \text{or}\ 
r=s\ \text{and}\ i<j.$$

The free differential perm algebra generated by $X$, denoted by $\Der\Perm\langle X\rangle$, can be realized as the free perm algebra $\Perm\langle X^{(\omega)}\rangle$ equipped with the derivation determined by 

$$d\bigl(x_i^{(s)}\bigr)=x_i^{(s+1)}$$

and the Leibniz rule  $d(uv)=d(u)v+u\,d(v),$ see \cite{KolesnikovSartayev2022}.

Hence, as a perm algebra, 
$$\Der\Perm\langle X\rangle \cong \Perm\langle X^{(\omega)}\rangle.$$ 
It follows that a linear basis of $\Der\Perm\langle X\rangle$ consists of the monomials 
\begin{equation}\label{eq:derperm-basis}    x_{k_1}^{(i_1)} x_{k_2}^{(i_2)} \cdots x_{k_{n-1}}^{(i_{n-1})} x_{k_n}^{(i_n)}, \qquad n\geq1,
\end{equation} 
such that
 $$x_{k_1}^{(i_1)} \leq x_{k_2}^{(i_2)} \leq\cdots\leq x_{k_{n-1}}^{(i_{n-1})},$$ where the last factor $x_{k_n}^{(i_n)}$ is arbitrary. We denote by $B_n$ the set of monomials \eqref{eq:derperm-basis} of length $n$, and set $B:=\bigcup_{n\ge 1} B_n.$

The degree $\deg(u)$ of a basis monomial \eqref{eq:derperm-basis} is its length $n$ and the weight $\wt(\,\cdot\,)$ is defined recursively by the rules
\[
\wt(x)=-1,\quad x\in X;\qquad
\wt(d(u))=\wt(u)+1;\qquad
\wt(uv)=\wt(u)+\wt(v).
\]
In particular, for the generators $x^{(s)}=d^s(x)$ one has
\[
\wt\!\bigl(x^{(s)}\bigr)=\wt(x)+s=s-1.
\]
Hence, for a basis monomial
\[
u=x_{k_1}^{(i_1)}\cdots x_{k_{n-1}}^{(i_{n-1})}x_{k_n}^{(i_n)},
\]
we obtain the explicit formula
\begin{equation}\label{eq:weight-formula}
\wt(u)=\sum_{j=1}^n \wt\!\bigl(x_{k_j}^{(i_j)}\bigr)=\sum_{j=1}^n (i_j-1)
=\Bigl(\sum_{j=1}^n i_j\Bigr)-n.
\end{equation}

These notations will be used below to describe the subalgebras of $\Der\Perm\langle X\rangle$ generated by $X$ with respect to the products $\blacklozenge$ and $\bullet$.

For an algebra $A$, denote its right annihilator by
\[
Ann_R(A):=\{a\in A\mid ab=0 \text{ for all } b\in A\}.
\]
Thus,
\[
f\notin Ann_R\bigl(\Der\Perm\langle X\rangle\bigr)
\]
means that there exists $u\in\Der\Perm\langle X\rangle$ such that
\[
fu\neq0.
\]

We now prove the main reduction result for differential polynomial identities.
\begin{theorem}\label{thm:reduction}
Let $(P,\cdot,d)$ be a differential perm algebra over a field of
characteristic zero. Suppose that $P$ satisfies a differential polynomial
identity $f$ such that
\[
f\notin
Ann_R\bigl(\Der\Perm\langle X\rangle\bigr).
\]
Then $P$ satisfies a differential polynomial identity of the form
\[
a_1'a_2'\cdots a_m'=0,
\]
for all $a_1,\ldots,a_m\in P$ and some $m\geq2$.
\end{theorem}

\begin{proof}
The argument follows the reduction used in
\cite{DotsenkoIsmailovUmirbaev2023} for commutative differential
algebras, with an additional step needed to take the right annihilator
of the free perm algebra into account.

Since the ground field has characteristic zero, we may assume that
$f$ is multilinear. Indeed, we may first choose a multihomogeneous
component of $f$ which does not belong to the right annihilator and
then polarize it. 
Write
\[
f=f(x_1,\ldots,x_k).
\]

We prove the statement by induction on $k$. More precisely, we prove by
induction on $k$ the following claim: if $P$ satisfies an
identity $f(x_1,\ldots,x_k)\,u=0$, where $f$ is multilinear such that
$f\notin Ann_R\bigl(\Der\Perm\langle x_1,\ldots,x_k\rangle\bigr)$,
and $u$ is a monomial in variables not occurring in $f$, then $P$ satisfies an
identity of the form
\[
x_1'\cdots x_p'x_{p+1}\cdots x_{p+q}=0 .
\]
 For $k=0$ the polynomial $f$ is a nonzero
scalar, so $P$ satisfies $u=0$, which is already of the required form.

Since $f$ is multilinear in $x_k$, every monomial occurring in $f$
contains precisely one factor $x_k^{(s)}$ for some $s\geq0$.
According to the basic monomials \eqref{eq:derperm-basis} described
above, this factor is either in the last position or among the
preceding factors. Hence $f$ can be written in the form
\begin{equation}\label{eq:f-decomposition}
f(x_1,\ldots,x_k)
=
\sum_{s=0}^{N} g_s x_k^{(s)}
+
\sum_{s=0}^{N} x_k^{(s)}p_s,
\end{equation}
where $N$ is the maximal order of a derivative of $x_k$ occurring in
$f$, and
\[
g_s,p_s\in
\Der\Perm\langle x_1,\ldots,x_{k-1}\rangle.
\]

Multiplying \eqref{eq:f-decomposition} on the right by a new variable
$u$ and using the identity \eqref{perm id}, we obtain
\begin{equation}\label{eq:fu}
f(x_1,\ldots,x_k)u
=
\sum_{s=0}^{N}(g_s+p_s)x_k^{(s)}u.
\end{equation}
Set
\[
c_s:=g_s+p_s.
\]

Since
\[
f\notin
Ann_R\bigl(\Der\Perm\langle X\rangle\bigr),
\]
the polynomial $fu$ is nonzero. Indeed, if $fu=0$ for the new
variable $u$, then substituting an arbitrary element for $u$ would
imply that $f$ belongs to the right annihilator.

Consequently, not all coefficients $c_s$ belong to the right
annihilator. Let $n$ be maximal such that
\[
c_n\notin
Ann_R\bigl(
\Der\Perm\langle x_1,\ldots,x_{k-1}\rangle
\bigr).
\]
All terms with $s>n$ vanish after multiplication on the right.
Therefore \eqref{eq:fu} may be written as
\[
fu
=
\sum_{s=0}^{n}c_sx_k^{(s)}u,
\qquad
c_n\notin
Ann_R\bigl(
\Der\Perm\langle x_1,\ldots,x_{k-1}\rangle
\bigr).
\]

If $n=0$, it is trivial.
Thus, it remains to consider the case $n\geq1$.

Consider the polynomial
\[
h_0(x_1,\ldots,x_{k-1},y,z)
:=
f(x_1,\ldots,x_{k-1},y)z
-
f(x_1,\ldots,x_{k-1},z)y.
\]
Using \eqref{eq:f-decomposition} and the identity
\eqref{perm id}, we obtain
\[
h_0
=
\sum_{s=0}^{N}
c_s\bigl(y^{(s)}z-z^{(s)}y\bigr).
\]
After multiplication on the right by $u$, the term corresponding to
$s=0$ vanishes, since
\[
(yz-zy)u=0.
\]
Moreover, all terms for which $c_s$ belongs to the right annihilator
vanish after right multiplication. Hence
\begin{equation}\label{eq:h0u}
h_0u
=
\sum_{s=1}^{n}
c_s\bigl(y^{(s)}z-z^{(s)}y\bigr)u,
\end{equation}
where
\[
c_n\notin
Ann_R\bigl(
\Der\Perm\langle x_1,\ldots,x_{k-1}\rangle
\bigr).
\]

We now reduce the maximal derivative order. Let $t$ be a new variable
and define
\[
\begin{aligned}
h_1(y,z,t)u
:={}&
h_0(yt,z)u-h_0(y,z)tu\\
&-h_0(zt,y)u+h_0(z,y)tu.
\end{aligned}
\]
By the Leibniz rule,
\[
(yt)^{(s)}
=
\sum_{i=0}^{s}
\binom{s}{i}
y^{(s-i)}t^{(i)}.
\]
Using the identity \eqref{perm id}, we obtain
\[
h_0(yt,z)u-h_0(y,z)tu
=
\sum_{s=1}^{n}
c_s
\sum_{i=1}^{s}
\binom{s}{i}
y^{(s-i)}t^{(i)}zu.
\]
Similarly,
\[
h_0(zt,y)u-h_0(z,y)tu
=
\sum_{s=1}^{n}
c_s
\sum_{i=1}^{s}
\binom{s}{i}
z^{(s-i)}t^{(i)}yu.
\]
Therefore
\[
\begin{aligned}
h_1(y,z,t)u
&=
\sum_{s=1}^{n}
c_s
\sum_{i=1}^{s}
\binom{s}{i}
t^{(i)}
\bigl(y^{(s-i)}z-z^{(s-i)}y\bigr)u
\\
&=
\sum_{s=1}^{n-1}
\sum_{i=1}^{n-s}
c_{s+i}
\binom{s+i}{i}
t^{(i)}
\bigl(y^{(s)}z-z^{(s)}y\bigr)u.
\end{aligned}
\]
Thus the maximal derivative order of $y$ and $z$ has decreased from
$n$ to $n-1$. In particular, the coefficient of
\[
\bigl(y^{(n-1)}z-z^{(n-1)}y\bigr)u
\]
is
\[
nc_nt'.
\]

Repeating this procedure with new variables
$t_1,\ldots,t_{n-1}$, we obtain the identity
\begin{equation}\label{eq:hn-1}
h_{n-1}(y,z,t_1,\ldots,t_{n-1})u
=
n!c_nt_1'\cdots t_{n-1}'
(y'z-yz')u.
\end{equation}

Finally, substituting $yt_n$ for $y$ in
\eqref{eq:hn-1}, we obtain
\[
\begin{aligned}
&h_{n-1}(yt_n,z,t_1,\ldots,t_{n-1})u
-
h_{n-1}(y,z,t_1,\ldots,t_{n-1})t_nu
\\
&\qquad
=
n!c_nt_1'\cdots t_{n-1}'t_n' yzu.
\end{aligned}
\]
Hence $P$ satisfies the identity
\[
c_nt_1'\cdots t_n' yzu=0.
\]

The coefficient $c_n$ depends only on
$x_1,\ldots,x_{k-1}$ and
\[
c_n\notin
Ann_R\bigl(
\Der\Perm\langle x_1,\ldots,x_{k-1}\rangle
\bigr).
\]
Therefore the  induction hypothesis is applicable to
$c_n$, with the fixed right factor
\[
t_1'\cdots t_n' yzu.
\]
It follows that $P$ satisfies an identity of the form
\[
x_1'\cdots x_p'
x_{p+1}\cdots x_{p+q}=0
\]
for some $p,q\geq0$ with $p+q\geq2$.

Finally, applying the substitutions
\[
x_{p+i}\longmapsto x_{p+i}',
\qquad
1\leq i\leq q,
\]
we obtain
\[
x_1'x_2'\cdots x_m'=0,
\qquad
m=p+q\geq2.
\]
This completes the proof.
\end{proof}


\section{\texorpdfstring{Symmetric Elements and the $\blacklozenge$-Product in $\Der\Perm\langle X\rangle$}{Symmetric Elements and the blacklozenge-product in $\Der\Perm\langle X\rangle$}}

In this section, we study the subspace $ST\langle X\rangle\subseteq \Der\Perm\langle X\rangle$ generated from $X$ by the operation $a\blacklozenge b=ab'+ba'$.

A commutative algebra $(A,\star)$ is called a {Tortken algebra} if it satisfies the
 identity
\begin{equation}\label{eq:tortken1}
 (a\star b)\star (c\star d)
-(a\star d)\star (c\star b)
+\langle a,b,c\rangle_\star \star d
-\langle a,d,c\rangle_\star \star b=0,
\end{equation}
for all $a,b,c,d\in A$, where $\langle a,b,c\rangle_\star:=(a\star b)\star c-a\star(b\star c)$ is the associator.

Indeed, if $(A,\cdot,d)$ is a commutative associative differential algebra and one defines
\begin{equation}\label{eq:symmdiff}
a\star b := d(a)b+a\,d(b),
\end{equation}
then $(A,\star)$ is a commutative algebra which satisfies the identity
\eqref{eq:tortken1} and an additional identity in degree five  \cite{Dzhumadil2002,Dzhumadil2005}.

For a compact form of the latter identity, put
\[
 L(x_1,x_2,x_3,x_4,x_5)
 :=((((x_1\blacklozenge x_2)\blacklozenge x_3)
       \blacklozenge x_4)\blacklozenge x_5).
\]
The group $S_5$ acts by permuting the variables of $L$.  In one-line notation
for permutations, define the following element of the group algebra
$\mathbb K S_5$:
\begin{align*}
 \Omega={}&2(21345)-(21354)-2(21435)+(21453)\\&-(31245)+(31254)+(31452)-(31524)\\
 &+(41235)-(41253)-(41352)+(41523)\\&+(51234)-(51243).
\end{align*}
 
 Although perm algebras are not commutative in general, $\Der\Perm\langle X\rangle$ under the symmetrized differential product 
$\blacklozenge$ 
satisfies the Tortken identity and an additional identity of degree $5$, as we see in the next proposition.

\begin{proposition}\label{prop:lozenge-identities}
For all $a,b,c,d,e\in\Der\Perm\langle X\rangle$, the product
$\blacklozenge$ is commutative, satisfies the Tortken identity
\eqref{eq:tortken1}, and satisfies the degree-five identity
\begin{equation}\label{eq:deg5}
 (\Omega\cdot L)(a,b,c,d,e)=0.
\end{equation}
\end{proposition}

\begin{Proof}
Expanding the product $\blacklozenge$, applying the Leibniz rule, and reducing the resulting monomials gives the stated identities. To achieve this, we use Wolfram Mathematica. A special code has been developed for calculating all substitutions \cite{Code}.
\end{Proof}

Denote by $ST\langle X\rangle$ the smallest linear subspace of $\Der\Perm\langle X\rangle$
containing $X$ and closed under the symmetrized product
\[
a\blacklozenge b:=ab'+ba'.
\]

To describe this subalgebra, we introduce a linear operator $^*: \Der\Perm\langle X\rangle\to \Der\Perm\langle X\rangle$ defined recursively by
\[
\bigl(x^{(i)}\bigr)^{*}=d\bigl(x^{(i)}\bigr)=x^{(i+1)},\qquad (uv)^{*}=u\,v^{*}+v\,u^{*},
\]
for $x^{(i)}\in X^{(\omega)}$ and $u,v\in \Der\Perm\langle X\rangle$.

Indeed, for all $n\geq1$ and
$a_1,\ldots,a_n\in X^{(\omega)}$,
\begin{equation}\label{eq:star-explicit}
(a_1a_2\cdots a_n)^*
=
\sum_{j=1}^{n}
a_1\cdots a_{j-1}a_{j+1}\cdots a_n\,a_j^*.
\end{equation}
Thus, in each summand, one factor is acted on by $^*$ and moved to the
distinguished last position.

\begin{Proposition}\label{prp: *}

For every $n\geq1$, all $a_1,\ldots,a_n\in X^{(\omega)}$, and every
$\sigma\in S_n$, we have
\begin{equation}\label{eq u*=u_g*}
(a_1a_2\cdots a_n)^*
=
(a_{\sigma(1)}a_{\sigma(2)}\cdots a_{\sigma(n)})^*.
\end{equation}
Moreover,
\begin{equation}\label{eq u**=u'*}
(u^*)^*=(u')^*=(u^*)'
\end{equation} and \begin{equation}\label{prp: wt(u^*)}
\wt(u^*)=\wt(u)+1.
\end{equation}
for every monomial $u\in\Der\Perm\langle X\rangle$.
\end{Proposition}

\begin{Proof} The proof of \eqref{eq u*=u_g*} follows immediately from
\eqref{eq:star-explicit} by induction on $n$. 
The second equality follows from \eqref{eq u*=u_g*} and Leibniz rule. Every summand in
\eqref{eq:star-explicit} has weight $\wt(u)+1$, which proves
\eqref{prp: wt(u^*)}. 
\end{Proof}

Let $C\subseteq B$ be defined by
\[
  C:=\bigcup_{n\geq1}C_n, \text{ where }C_n:=\left\{
 x_{k_1}^{(i_1)}\cdots x_{k_n}^{(i_n)}
 \,\middle|\,
 x_{k_1}^{(i_1)}\leq\cdots\leq x_{k_n}^{(i_n)}
 \right\}.
\]
Then every orbit of a basis monomial under permutation of all its factors has a
unique representative in $C$.

Define
\[
S := X \cup \{\, u^* \mid u\in C,\ \mathrm{wt}(u) = -2 \}.
\]
We call the elements of the linear span $\langle S\rangle$ symmetric elements.
The following lemma shows that $\langle S\rangle$ is closed under the
product $\blacklozenge$.

\begin{lemma}\label{lem: S closed}
The linear span $\langle S\rangle$ is closed under the product $a\blacklozenge b = ab' + ba'.$
\end{lemma}
\begin{proof}
Since $\blacklozenge$ is bilinear, it suffices to check closure for generators of $S$.

\medskip
\noindent

Let $x,y\in X$. Then
\[
x\blacklozenge y = xy' + yx'.
\]
By the definition of the $^*$,
\[
(xy)^* = xy' + yx'.
\]
Since
\[
\mathrm{wt}(xy)=\mathrm{wt}(x)+\mathrm{wt}(y)=-1-1=-2,
\]
we obtain $(xy)^*\in S$. Hence
\[
x\blacklozenge y\in \langle S\rangle.
\]

Let $u^*\in S$ with $\mathrm{wt}(u)=-2$, and let $x\in X$. Then, using the \eqref{eq u**=u'*}, we obtain

\[
u^*\blacklozenge x = u^*x' + x(u^*)^*
=(u^*x)^*.
\]

Since $x\in X$, we have $\mathrm{wt}(x)=-1$, and by \eqref{prp: wt(u^*)}

\[
\mathrm{wt}(u^*x)=\mathrm{wt}(u^*)+\mathrm{wt}(x)=(-1)+(-1)=-2.
\]

Every basis monomial occurring in $u^*x$ has weight $-2$.
After permuting all its factors to the unique representative in $C$,
Proposition~\ref{prp: *} shows that its $^*$-image belongs to $S$.
Hence $u^*\blacklozenge x\in \langle S\rangle$.

Let $u^*,v^*\in S$. Then
\[
u^*\blacklozenge v^*
= u^*(v^*)^* + v^*(u^*)^*
= (u^*v^*)^*.
\]

By \eqref{prp: wt(u^*)}, we have 
\[
\mathrm{wt}(u)=\mathrm{wt}(v)=-2,
\qquad
\mathrm{wt}(u^*)=\mathrm{wt}(v^*)=-1,
\]

Therefore  $\mathrm{wt}(u^*v^*)=-2$ and
$u^*\blacklozenge v^*\in \langle S\rangle .$
Thus, $\langle S\rangle$ is closed under $\blacklozenge$.
\end{proof}

The preceding lemma shows that \[ ST\langle X\rangle\subseteq \langle S\rangle. \] To prove the reverse inclusion, it remains to show that every element in  $S$ can be obtained from $X$ by iterating the product $\blacklozenge$.

\begin{theorem}\label{th: S subseteq ST}
Let  $w\in B$ be a basis monomial of weight $\wt(w)=-2$, then $w^{*}\in ST\langle X\rangle$. In particular, $\langle S\rangle\subseteq ST\langle X\rangle$.
\end{theorem}

\begin{proof}
We will prove it by induction on $n=\deg(w)$. If $n=2$, then $w=xy$ for some $x,y\in X$, and by the definition of $^{*}$ we have
\[
w^{*}=(xy)^{*}=xy'+yx'=x\blacklozenge y\in ST\langle X\rangle.
\]
If $n=3$ and $\wt(w)=-2$, then $w$ has exactly one first derivative and no higher derivatives. Since $w^*$ is invariant under permutations of all factors by
Proposition~\ref{prp: *}, we may replace $w$ by a permutation of its
factors and assume that
$w=xyz'.$
In this case, the identity
\begin{equation}\label{eq:base-xyz-rewrite}
(xyz')^{*}
=\frac{1}{2}\Big( (x\blacklozenge z)\blacklozenge y+(y\blacklozenge z)\blacklozenge x-(x\blacklozenge y)\blacklozenge z\Big)
\end{equation}
shows directly that $(xyz')^{*}\in ST\langle X\rangle$.

Assume $n\ge 4$ and that $w^{*}\in ST\langle X\rangle$ holds for every basis monomial $w\in B$
with $\deg(w)<n$ and $\wt(w)=-2$. Let $w\in B$ satisfy $\deg(w)=n$ and $\wt(w)=-2$.
We distinguish two cases.

Suppose first that $w$ contains a first derivative $x_i'$.  Since $\wt(w)=-2$, by Proposition \ref{prp: *} we can choose a factor $x_j\in X$
and  a monomial $w_0$ such that
\[
w=x_i' x_j w_0,\qquad \wt(w_0)=-1.
\]
By the definition of the operator $^*$ and the left-commutative identity,  we have
\begin{equation}\label{eq:star-lemma31-rewrite}
\begin{aligned}
(x_i'x_j w_0)^{*}
&=\frac12\Big( \big((x_ix_j)^{*}w_0\big)^{*}
+\big((x_i w_0)^{*}x_j\big)^{*}
-\big((x_j w_0)^{*}x_i\big)^{*}\Big).
\end{aligned}
\end{equation}

First, $\deg(x_i w_0)=\deg(x_j w_0)=n-1<n$, and
\[
\wt(x_i w_0)= \wt(x_j w_0)=-2.
\]
Hence, by the induction hypothesis,
\begin{equation}\label{eq:(x_i w_0)^{*} in ST<X>}
(x_i w_0)^{*},\ (x_j w_0)^{*}\in ST\langle X\rangle.
\end{equation}
Therefore by \eqref{eq u**=u'*} we have
\[
\bigl((x_i w_0)^*x_j\bigr)^*=
(x_i w_0)^*x_j'
+
x_j\bigl((x_i w_0)^*\bigr)^*=
(x_i w_0)^*\blacklozenge x_j
\in ST\langle X\rangle.
\]
Similarly,
\[
\bigl((x_j w_0)^*x_i\bigr)^*
=
(x_j w_0)^*\blacklozenge x_i
\in ST\langle X\rangle.
\]

Let $t$ be a new generator. By the definition of the operator $^*$, we have
\[
(tw_0)^*=tw_0^*+w_0t'.
\] Since
\[
\deg(tw_0)=n-1
\qquad\text{and}\qquad
\wt(tw_0)=-2,
\] the induction hypothesis, applied to the enlarged set of generators
$X\cup\{t\}$, gives
\begin{equation}\label{tw* in ST}
(tw_0)^*\in ST\langle X\cup\{t\}\rangle.
\end{equation}
Hence, $(tw_0)^*$ can be represented by a polynomial in the product
$\blacklozenge$ with generators from $X\cup\{t\}$.

Now substitute
\[
t\mapsto (x_ix_j)^*
=x_i\blacklozenge x_j\in ST\langle X\rangle.
\]
Then \eqref{tw* in ST} becomes
\[
(x_ix_j)^*w_0^*
+
w_0\bigl((x_ix_j)^*\bigr)' \in ST\langle X\rangle.
\]
By \eqref{eq u**=u'*}, we have
\[
\bigl((x_ix_j)^*\bigr)'
=
\bigl((x_ix_j)^*\bigr)^*.
\]
Consequently,
\[
\begin{aligned}
(x_ix_j)^*w_0^*
+
w_0\bigl((x_ix_j)^*\bigr)'
&=
(x_ix_j)^*w_0^*
+
w_0\bigl((x_ix_j)^*\bigr)^*\\
&=
\bigl((x_ix_j)^*w_0\bigr)^*.
\end{aligned}
\]

Therefore,
\[
\bigl((x_ix_j)^*w_0\bigr)^*
\in ST\langle X\rangle.
\]

Thus all three summands in \eqref{eq:star-lemma31-rewrite} lie in $ST\langle X\rangle$, and hence
$w^{*}\in ST\langle X\rangle.$

It remains to consider the case in which no first-order factor occurs. 
Assume that $w$ has no first derivative factors of the form $x_i'$. Since $\wt(w)=-2$, this means that
at least one factor has derivation order $\ge 2$. By Proposition~\ref{prp: *} we may write
\[
w=w_1w_2,\qquad
w_1=x_{j_1}\cdots x_{j_k}x_i^{(k)}\quad\text{with }k>1,
\]
where $w_2$ is a differential monomial.

Define
\[
v:=x_{j_1}\cdots x_{j_k}x_i^{(k-1)}.
\]
Then expanding $v^{*}$ gives
\[
v^{*}=w_1+\sum_{\ell=1}^{k}x_i^{(k-1)}x_{j_1}\cdots \hat x_{j_\ell}\cdots x_{j_k}x_{j_\ell}'.
\]
By Proposition~\ref{prp: *} we obtain
\begin{equation}\label{eq:case2-star-rewrite}
w^{*}=(w_1w_2)^{*}
=\bigl(v^{*}w_2\bigr)^{*}
-\sum_{\ell=1}^{k}\Bigl(x_i^{(k-1)}x_{j_1}\cdots x_{j_\ell}'\cdots x_{j_k}\,w_2\Bigr)^{*}.
\end{equation}

By induction, $v^*\in ST\langle X\rangle$.  Also
$(tw_2)^*\in ST\langle X\cup \{t\}\rangle$ for a new generator $t$, because
$\wt(tw_2)=-2$ and $\deg(tw_2)<n$.  Substituting $v^*$ for $t$ shows that
the first term of \eqref{eq:case2-star-rewrite} belongs to
$ST\langle X\rangle$.  Every term in the sum contains a first derivative and
is covered by the previous case.
 
 This completes the proof.
\end{proof}

Combining Lemma \ref{lem: S closed} and Theorem \ref{th: S subseteq ST}, we can describe the multilinear components explicitly.

\begin{corollary} The set $S$ is a basis of $ST\langle X\rangle$.
If $ST_n^{\mathrm{mult}}$ is the multilinear component in
$x_1,\ldots,x_n$, then
\[
 \dim ST_n^{\mathrm{mult}}=\binom{2n-3}{n-1}
 \qquad(n\geq2).
\]
\end{corollary}
\begin{proof}
By Lemma~\ref{lem: S closed} and Theorem~\ref{th: S subseteq ST}, the set $S$ spans $ST\langle X\rangle$. It is clear that the elements in $S$ are linearly independent.
 In the multilinear component of degree $n$, a monomial
$u\in C$ of weight $-2$ is determined uniquely by the nonnegative derivative orders
$i_1,\ldots,i_n$ assigned to $x_1,\ldots,x_n$. By \eqref{eq:weight-formula},
\[
i_1+\cdots+i_n=n-2.
\]
The number of weak compositions of $n-2$ into $n$ parts is
\[
\binom{(n-2)+n-1}{n-1}
=
\binom{2n-3}{n-1}.
\]
This proves the dimension formula.
\end{proof}

\section{\texorpdfstring{Weak-symmetric Elements and the $\bullet$-Product in $\Der\Perm\langle X\rangle$}{Weak-symmetric Elements and the bullet-product in DerPerm<X>}}

In this section, we study the subspace $\overline{ST}\langle X\rangle\subseteq \Der\Perm\langle X\rangle$
generated from $X$ by the operation $a\bullet b=a'b+ab'$.

\begin{Proposition}\label{prop:bullet-diidentities}
Let $\Der\Perm\langle X\rangle$ be the free differential perm algebra with derivation $d$, and endow it
with the bilinear product
\[
a\bullet b:=a'b+ab'.
\]
Then, for all $a,b,c,d\in\Der\Perm\langle X\rangle$, the product $\bullet$ satisfies the left-commutative identity
\begin{equation}\label{eq:left-comm-bullet}
(a\bullet b)\bullet c=(b\bullet a)\bullet c,
\end{equation}
and the following two Tortken-type identities:
\begin{equation}\label{eq:tortken-di-1}
\begin{aligned}
&(a\bullet b)\bullet (c\bullet d)-(c\bullet b)\bullet(a\bullet d)
+\Bigl((a,b,c)_{\bullet}\Bigr)\bullet d
+b\bullet\Bigl(a\bullet(c\bullet d)-c\bullet(a\bullet d)\Bigr)=0,
\end{aligned}
\end{equation}
\begin{equation}\label{eq:tortken-di-2}
\begin{aligned}
&(a\bullet b)\bullet (c\bullet d)-(a\bullet c)\bullet (b\bullet d)
-b\bullet\Bigl((a,c,d)_{\bullet}\Bigr)
+c\bullet\Bigl((a,b,d)_{\bullet}\Bigr)=0,
\end{aligned}
\end{equation}
where $(x,y,z)_{\bullet}=(x\bullet y)\bullet z-x\bullet(y\bullet z)$.
\end{Proposition}

\begin{Proof}
Expanding the $\bullet$-products, applying the Leibniz rule, and using the identity \eqref{perm id},
we obtain the stated identities by a direct computation. The same expansions are verified in the supplementary
Mathematica code \cite{Code}.
\end{Proof}

Define
\[
\overline{S} := X \cup \{\, u' \mid u\in B,\ \mathrm{wt}(u) = -2 \,\}.
\]
We refer to the elements of the linear span $\langle \overline{S} \rangle$ as
\textit{weak-symmetric elements}.

The main difference from the symmetric elements considered in the previous section is that the ordinary derivative $u'$ is not invariant under permuting all factors. Indeed, by Proposition~\ref{prp: *}, the operator $^*$ satisfies \eqref{eq u*=u_g*}, whereas in general
\[
(x_1\cdots x_n)'\neq (x_1\cdots x_{n-2}x_nx_{n-1})'.
\]
Thus, in contrast to the definition of $S$, where it was enough to take the representatives from $C$, here we retain all basis monomials $u\in B$ of weight $-2$.

The following lemma is the analogue of Lemma~\ref{lem: S closed} and shows that this family of elements is closed under the product $\bullet$.

\begin{lemma}\label{lem: {S} closed}
The linear span $\langle \overline{S} \rangle$ is closed under the product
$a \bullet b = a'b + ab'$.
\end{lemma}

\begin{proof}
Since the product $\bullet$ is bilinear, it suffices to verify closure for the generators of $\overline{S}$.

Let $x,y\in X$. Then
\[
x\bullet y=x'y+xy'=(xy)'.
\]
Since $\wt(xy)=-2$, we have $(xy)'\in\overline{S}$, and hence
$x\bullet y\in\langle\overline{S}\rangle$.

Let $u'\in\overline{S}$ with $\wt(u)=-2$, and let $x\in X$. Then
\[
u'\bullet x=(u')'x+u'x'=(u'x)',
\qquad
x\bullet u'=x'u'+x(u')'=(xu')'.
\]
Every basis monomial occurring in $u'$ has weight $-1$. Therefore every basis monomial occurring in
$u'x$ and in $xu'$ has weight $-2$. By the definition of $\langle\overline{S}\rangle$ and the linearity of the derivation, we obtain
\[
(u'x)',\ (xu')'\in\langle\overline{S}\rangle.
\]
Hence both $u'\bullet x$ and $x\bullet u'$ belong to $\langle\overline{S}\rangle$.

Finally, let $u',v'\in\overline{S}$ with $\wt(u)=\wt(v)=-2$. Then
\[
u'\bullet v'=(u')'v'+u'(v')'=(u'v')'.
\]
Every basis monomial occurring in $u'v'$ has weight $-2$, because
$\wt(u')=\wt(v')=-1$. Therefore
\[
(u'v')'\in\langle\overline{S}\rangle.
\]
Thus, $\langle\overline{S}\rangle$ is closed under $\bullet$.
\end{proof}

The preceding lemma shows that
\[
\overline{ST}\langle X\rangle\subseteq \langle\overline{S}\rangle.
\]
To prove the reverse inclusion, it remains to show that every element of $\overline{S}$ can be obtained from $X$ by iterating the product $\bullet$.

As in the proof of Theorem~\ref{th: S subseteq ST}, we begin with the degree-three reductions. A direct computation gives, for all $a,b,c\in\Der\Perm\langle X\rangle$,
\begin{equation}\label{eq: di left exp}
(a'bc)'=\frac12\Big((a\bullet b)\bullet c+b\bullet(a\bullet c)-a\bullet(b\bullet c)\Big),
\end{equation}
and
\begin{equation}\label{eq: di right exp}
(abc')'=\frac12\Big(b\bullet(a\bullet c)+a\bullet(b\bullet c)-(a\bullet b)\bullet c\Big).
\end{equation}

\begin{theorem}\label{th: {S} subseteq {ST}}
Let $u\in B$ be a basis monomial with $\wt(u)=-2$. Then
$u'\in \overline{ST}\langle X\rangle$.
In particular, $\langle \overline{S}\rangle\subseteq \overline{ST}\langle X\rangle$.
\end{theorem}

\begin{proof}
We prove the statement by induction on $n=\deg(u)$.
If $n=2$, then $u=xy$ for some $x,y\in X$, and
\[
u'=(xy)'=x\bullet y\in\overline{ST}\langle X\rangle.
\]

If $n=3$ and $\wt(u)=-2$, then $u$ contains exactly one first derivative and no higher derivatives.
Using the identity \eqref{perm id}, we may reorder the first two factors while keeping the rightmost factor fixed. Hence, after relabeling the variables, $u$ is of one of the two forms
\[
u=x_1'x_2x_3
\qquad\text{or}\qquad
u=x_1x_2x_3'.
\]
The identities \eqref{eq: di left exp} and \eqref{eq: di right exp} show directly that
$u'\in\overline{ST}\langle X\rangle$.

Assume now that $n\ge4$ and that the statement holds for all basis monomials of degree less than $n$ and weight $-2$.
Let $u\in B$ satisfy $\deg(u)=n$ and $\wt(u)=-2$.
The first two reductions below are parallel to those in the proof of Theorem~\ref{th: S subseteq ST}. The only additional point is that, unlike $u^*$, the derivative $u'$ is not invariant under changing the distinguished rightmost factor.

\medskip
\noindent\textit{Case 1: a first derivative occurs among the first $n-1$ factors.}
Assume that a factor $x_i'$ occurs among the first $n-1$ factors of $u$.
Since $\wt(u)=-2$, formula \eqref{eq:weight-formula} gives total derivation order $n-2$.
Hence at least one of the first $n-1$ factors is undifferentiated. We may write
\[
u=x_i'x_j u_0,
\]
where $x_j\in X$, the rightmost factor of $u_0$ coincides with the rightmost factor of $u$, and
\[
\wt(u_0)=-1.
\]
By using \eqref{eq: di left exp} we have
\begin{equation}\label{eq:bullet-case1}
 u'=\frac12\Bigl((x_i\bullet x_j)\bullet u_0
 +x_j\bullet(x_i\bullet u_0)
 -x_i\bullet(x_j\bullet u_0)\Bigr).
\end{equation}
Now
\[
x_i\bullet u_0=(x_i u_0)',
\qquad
x_j\bullet u_0=(x_j u_0)'.
\]
After rewriting the first factors, $x_i u_0$ and $x_j u_0$ are basis monomials of degree $n-1$ and weight $-2$. Hence, by the induction hypothesis,
\[
(x_i u_0)',\ (x_j u_0)'\in\overline{ST}\langle X\rangle.
\]
Therefore the second and third terms in \eqref{eq:bullet-case1} belong to $\overline{ST}\langle X\rangle$.

For the first term, let $t$ be a new generator. Since
\[
\deg(tu_0)=n-1,
\qquad
\wt(tu_0)=-2,
\]
the induction hypothesis, applied to the enlarged set of generators $X\cup\{t\}$, gives
\[
(tu_0)'=t\bullet u_0\in\overline{ST}\langle X\cup\{t\}\rangle.
\]
Substituting
\[
t\mapsto x_i\bullet x_j\in\overline{ST}\langle X\rangle
\]
shows that
\[
(x_i\bullet x_j)\bullet u_0\in\overline{ST}\langle X\rangle.
\]
Thus every term in \eqref{eq:bullet-case1} belongs to $\overline{ST}\langle X\rangle$, and therefore
$u'\in\overline{ST}\langle X\rangle$.

\medskip
\noindent\textit{Case 2: no first derivative occurs among the first $n-1$ factors, but a higher derivative occurs there.}
Choose among the first $n-1$ factors a factor $x_i^{(k)}$ of minimal positive derivation order. Then $k\ge2$.
Let $q$ be the number of differentiated factors among the first $n-1$ factors. Since $k$ is minimal, we have $qk\le n-2$. Hence $q+k\le qk+1\le n-1$, so there are at least $k$ undifferentiated factors among the first $n-1$ factors. We may write
\[
u=w_1w_2,
\qquad
w_1=x_i^{(k)}x_{j_1}\cdots x_{j_k},
\]
where $x_{j_1},\ldots,x_{j_k}\in X$ and the rightmost factor of $w_2$ coincides with the rightmost factor of $u$.
Since $\wt(w_1)=-1$, we have $\wt(w_2)=-1$.

Set
\[
v:=x_i^{(k-1)}x_{j_1}\cdots x_{j_k}.
\]
Then $\wt(v)=-2$ and $\deg(v)=k+1<n$. Hence, by the induction hypothesis,
\[
v'\in\overline{ST}\langle X\rangle.
\]
Expanding $v'$ gives
\[
v'=w_1+
\sum_{\ell=1}^{k}
 x_i^{(k-1)}x_{j_1}\cdots x_{j_\ell}'\cdots x_{j_k}.
\]
Consequently,
\begin{equation}\label{eq:bullet-case2}
\begin{aligned}
u'
&=(v'w_2)'
-\sum_{\ell=1}^{k}
\Bigl(x_i^{(k-1)}x_{j_1}\cdots x_{j_\ell}'\cdots x_{j_k}w_2\Bigr)'.
\end{aligned}
\end{equation}
Every monomial in the sum contains a first derivative among the first $n-1$ factors, so Case~1 applies.

For the first term, let $t$ again be a new generator. Since
\[
\wt(tw_2)=-2,
\qquad
\deg(tw_2)<n,
\]
the induction hypothesis gives
\[
(tw_2)'\in\overline{ST}\langle X\cup\{t\}\rangle.
\]
Substituting $t\mapsto v'$ yields
\[
(v'w_2)'\in\overline{ST}\langle X\rangle.
\]
Hence \eqref{eq:bullet-case2} implies $u'\in\overline{ST}\langle X\rangle$.

\medskip
\noindent\textit{Case 3: the first $n-1$ factors are undifferentiated.}
In this case, \eqref{eq:weight-formula} implies that the rightmost factor has derivation order $n-2$.
Put $r=n-2$. After reordering the first $n-1$ factors, we may write
\[
u=x_{j_0}x_{j_1}\cdots x_{j_r}x_i^{(r)}.
\]
Set
\[
v:=x_{j_1}\cdots x_{j_r}x_i^{(r-1)}.
\]
Then $\deg(v)=n-1$ and $\wt(v)=-2$, so the induction hypothesis gives
\[
v'\in\overline{ST}\langle X\rangle.
\]
Moreover,
\[
v'=x_{j_1}\cdots x_{j_r}x_i^{(r)}
+\sum_{\ell=1}^{r}
 x_{j_1}\cdots x_{j_\ell}'\cdots x_{j_r}x_i^{(r-1)}.
\]
Therefore
\begin{equation}\label{eq:bullet-case3}
\begin{aligned}
u'
&=(x_{j_0}v')'
-\sum_{\ell=1}^{r}
\Bigl(x_{j_0}x_{j_1}\cdots x_{j_\ell}'\cdots x_{j_r}x_i^{(r-1)}\Bigr)'.
\end{aligned}
\end{equation}
The first term belongs to $\overline{ST}\langle X\rangle$, because
\[
(x_{j_0}v')'=x_{j_0}\bullet v'
\]
and $x_{j_0}\in X$, $v'\in\overline{ST}\langle X\rangle$.
Every monomial in the sum has a first derivative among the first $n-1$ factors, and hence Case~1 applies.
Thus \eqref{eq:bullet-case3} gives $u'\in\overline{ST}\langle X\rangle$.

This completes the induction.
\end{proof}

Combining Lemma~\ref{lem: {S} closed} and Theorem~\ref{th: {S} subseteq {ST}}, we obtain
\[
\overline{ST}\langle X\rangle=\langle\overline{S}\rangle.
\]
Denote by $\overline{ST}_n$ the homogeneous component of $\overline{ST}\langle X\rangle$ of degree $n$, and by $\overline{ST}^{\mathrm{mult}}_n$ its multilinear part in $x_1,\ldots,x_n$.

\begin{corollary}
The set $\overline{S}$ is a basis of the space of weak-symmetric elements
$\overline{ST}\langle X\rangle$.
Moreover,
\[
\dim \overline{ST}^{\mathrm{mult}}_n
=n\binom{2n-3}{n-1},
\qquad n\ge2.
\]
\end{corollary}

\begin{proof}
The linear independence of $\overline{S}$ is clear. In the multilinear component of degree $n$, a basis monomial $u\in B$ of weight $-2$ is determined by the choice of its distinguished rightmost variable and by nonnegative integers $i_1,\ldots,i_n$ satisfying, by \eqref{eq:weight-formula},
\[
i_1+\cdots+i_n=n-2.
\]
There are $n$ choices for the rightmost variable and
$\binom{2n-3}{n-1}$ such $n$-tuples. Therefore
\[
\dim \overline{ST}^{\mathrm{mult}}_n
=n\binom{2n-3}{n-1}.
\]
\end{proof}

\section{Perm algebras and Leibniz brackets}

In this section, we show that every perm algebra endowed with the product $\circ$ gives rise to a Leibniz algebra. We then extend the construction to an arbitrary family of commuting derivations $D_1,\dots,D_n$ and prove that the resulting modified bracket still satisfies the left Leibniz identity. This provides a convenient way to build noncommutative Poisson-type structures starting from perm algebras.

Recall that a linear map $D:P\to P$ is called a $\delta$–derivation if
\[
D(xy)=\delta (D(x)y+xD(y))
\]
for all $x,y\in P$.

\begin{Proposition}\label{prop:perm_leibniz}
Let $D$ be a $\delta$–derivation of a perm algebra $(P,\cdot)$. 
Define a new multiplication
\[
a\circ b := D(a)b - aD(b), \qquad a,b\in P.
\]
Then $(P,\circ)$ is a Leibniz algebra, that is,
\[
(a\circ b)\circ c = a\circ (b\circ c) - b\circ (a\circ c)
\]
for all $a,b,c\in P$.
\end{Proposition}

\begin{Proof} Let $a,b,c \in P$ and write  $x'=D(x)$  for any $x \in P$. Since $(a'b - ab')' =\delta (a''b - ab'')$, a direct calculation gives
$$\begin{array}{lcl}
a\circ(b\circ c)-b\circ(a\circ c)&=& a'b'c-a'bc'-\delta ab''c+ \delta abc''-b'a'c+b'ac'+\delta ba''c-\delta bac''\\
& = & \delta a''bc-\delta ab''c-a'bc'+ab'c'\\
& = & (a\circ b)\circ c.
\end{array}$$
Thus
\[
(a\circ b)\circ c
=
a\circ(b\circ c)-b\circ(a\circ c),
\]
which proves the Leibniz identity.
\end{Proof}

We extend the above construction to an arbitrary family of commuting derivations $D_1,\dots,D_n$ and show that the modified bracket still satisfies the Leibniz identity.

\begin{proposition}\label{prop:leibniz-new-bracket}
Let $(P,\cdot)$ be a perm algebra.
Let $D_1,\dots,D_n\in\Der(P)$ be commuting derivations, $[D_i,D_j]=0$.
Set
\[
W_{\mathrm{perm}}:=\bigoplus_{i=1}^n P D_i,
\]
where each $P D_i$ is regarded as a labeled copy of $P$, so that the elements of $W_{\mathrm{perm}}$ are formal sums $\sum_{i=1}^n a_iD_i$ with $a_i\in P$.
Define a bilinear bracket on $W_{\mathrm{perm}}$ by
\begin{equation}\label{eq:new-leibniz-bracket}
[aD_i,\;bD_j]_{\circ}:=D_j(a)\cdot b\,D_i- a\cdot D_i(b)\,D_j,
\qquad a,b\in P,\ 1\leq i,j\leq n.
\end{equation}
Then $(W_{\text{perm}},[\ ,\ ]_{\circ})$ is a Leibniz algebra.
\end{proposition}

\begin{proof}
Fix $a,b,c\in P$ and $i,j,k\in\{1,\dots,n\}$. We expand the three terms of the Leibniz identity.

First,
\begin{align*}
[[aD_i,bD_j]_{\circ},cD_k]_{\circ}
&=[D_j(a)bD_i-aD_i(b)D_j,\;cD_k]_{\circ}\\
&=\big(D_kD_j(a)\,bc + D_j(a)D_k(b)c\big)\,D_i\\
&\quad-\big(D_k(a)D_i(b)c + aD_kD_i(b)c\big)\,D_j\\
&\quad-\big(D_j(a)bD_i(c)-aD_i(b)D_j(c)\big)\,D_k .
\end{align*}

Next,
\begin{align*}
[aD_i,[bD_j,cD_k]_{\circ}]_{\circ}
&=[aD_i,\;D_k(b)cD_j - bD_j(c)D_k]_{\circ}\\
&=\big(D_j(a)D_k(b)c - D_k(a)bD_j(c)\big)\,D_i\\
&\quad-\big(aD_iD_k(b)c + aD_k(b)D_i(c)\big)\,D_j\\
&\quad+\big(aD_i(b)D_j(c)+abD_iD_j(c)\big)\,D_k .
\end{align*}

Similarly,
\begin{align*}
[bD_j,[aD_i,cD_k]_{\circ}]_{\circ}
&=[bD_j,\;D_k(a)cD_i - aD_i(c)D_k]_{\circ}\\
&=-\big(bD_jD_k(a)c + bD_k(a)D_j(c)\big)\,D_i\\
&\quad+\big(D_i(b)D_k(a)c - D_k(b)aD_i(c)\big)\,D_j\\
&\quad+\big(bD_j(a)D_i(c)+baD_jD_i(c)\big)\,D_k .
\end{align*}




Comparing coefficients of $D_i$, $D_j$, and $D_k$ and using $[D_i,D_j]=[D_k,D_j]=[D_k,D_i]=0$, we obtain
\[
[[aD_i,bD_j]_{\circ},cD_k]_{\circ}=[aD_i,[bD_j,cD_k]_{\circ}]_{\circ}-[bD_j,[aD_i,cD_k]_{\circ}]_{\circ}.
\]
By bilinearity, the Leibniz identity holds in $(W_{\text{perm}}, [\ ,\ ]_{\circ})$.
\end{proof}

For $n=1$, taking $D_1=D$, formula \eqref{eq:new-leibniz-bracket} becomes
\[
[aD,bD]_{\circ}
=\bigl(D(a)b-aD(b)\bigr)D
=(a\circ b)D.
\]
Thus, Proposition~\ref{prop:leibniz-new-bracket} extends Proposition~\ref{prop:perm_leibniz} in the case of an ordinary derivation, that is, when $\delta=1$.

We next record the associated transposed Poisson-type compatibility.
Noncommutative Poisson algebras and their Leibniz analogues were studied in
\cite{Xu1994,CD-NLP}; transposed Poisson algebras were introduced in
\cite{BBGW}.

\begin{definition}
Let $\delta\in\mathbb{K}$. 
An algebra $(P,\cdot,\{\cdot,\cdot\})$ is called a noncommutative transposed $\delta$–Leibniz–
Poisson algebra if:
\begin{enumerate}
\item $(P,\cdot)$ is associative,
\item $(P,\{\cdot,\cdot\})$ is a Leibniz algebra,
\item the following identity holds: \begin{equation}\label{id: deltranspois}
    \delta x \{ y, z \} = \{ x y, z \} + \{ y, x z \}.
    \end{equation}
\end{enumerate}
\end{definition}

\begin{Proposition}\label{prop:delta_perm}
Let $\mathcal A$ be a perm algebra and let $D$ be a nontrivial $\delta$–derivation of $\mathcal A$. 
Define
\[
\{a,b\}:=D(a)b-aD(b).
\]
Then $(\mathcal A,\cdot,\{\cdot,\cdot\})$ is a noncommutative transposed $(\delta+1)$–Leibniz–Poisson algebra.
\end{Proposition}

\begin{Proof}
Since $\mathcal A$ is associative, the first condition is satisfied.
By Proposition~\ref{prop:perm_leibniz} we have $(\mathcal A,\{\cdot,\cdot\})$ is a Leibniz algebra. For the transposed compatibility condition, using the $\delta$–derivation property, we obtain
\begin{align*}
 \{xy,z\}+\{y,xz\}
 &=\delta D(x)yz+\delta xD(y)z-xyD(z)\\
 &\quad+D(y)xz-\delta yD(x)z-\delta yxD(z)\\
 &=(\delta+1)\bigl(xD(y)z-xyD(z)\bigr)\\
 &=(\delta+1)x\{y,z\}.
\end{align*}

Hence,  $(\mathcal A,\cdot,\{\cdot,\cdot\})$ is a noncommutative transposed $(\delta+1)$–Leibniz–Poisson algebra.
\end{Proof}

\section{Perm algebras and Witt-type constructions}\label{sec:perm-witt}

In \cite{KolesnikovSartayev2022} it was shown that, for a differential perm algebra, the product
\[
a\prec b:=ab' \qquad (b'=d(b))
\]
endows the underlying vector space with the structure of a left-symmetric (pre-Lie) algebra. Consequently, the operation
\[
a\Diamond b:=ab'-ba'=a\prec b-b\prec a=[a,b]
\]
is precisely the commutator bracket of this special left-symmetric algebra arising from the perm algebra with respect to the product $\prec$. It is well known that the commutator algebra of any left-symmetric algebra is a Lie algebra, i.e., the bracket is anti-commutative and satisfies the Jacobi identity. Therefore, $\Der\Perm\langle X\rangle$ equipped with the multiplication $\Diamond$ is a Lie algebra.

For $r\geq 1$, put
\begin{equation}\label{eq:standard-diamond}
\begin{aligned}
\operatorname{St}_{r+1}(x_1,\ldots,x_{r+1})
:=\sum_{\sigma\in S_r}\operatorname{sgn}(\sigma)\,
 x_{\sigma(1)}\Diamond\bigl(x_{\sigma(2)}\Diamond\cdots
 \bigl(x_{\sigma(r)}\Diamond x_{r+1}\bigr)\cdots\bigr).
\end{aligned}
\end{equation}

For a commutative differential associative algebra, the corresponding Lie algebra satisfies the standard identity $\operatorname{St}_5=0$. In contrast, $\Der\Perm\langle X\rangle$ under the multiplication $\Diamond$ does not satisfy the standard identity in degree five, but it satisfies the standard identity in degree six. Thus, passing from the commutative differential case to the differential perm case shifts the first standard identity from degree five to degree six.

\begin{Proposition}\label{pr: derperm lie}
The Lie algebra $\bigl(\Der\Perm\langle X\rangle,\Diamond\bigr)$ satisfies
\[
\operatorname{St}_6(x_1,\ldots,x_6)=0.
\]
Moreover, $\operatorname{St}_5$ is not an identity of $\bigl(\Der\Perm\langle X\rangle,\Diamond\bigr)$.
\end{Proposition}

\begin{Proof}
The anti-commutativity and the Jacobi identity follow from the fact that $\Diamond$ is the commutator of the left-symmetric product $a\prec b=ab'$.

It remains to prove the standard identity. Every monomial occurring in the expansion of
$\operatorname{St}_6(x_1,\ldots,x_6)$ has degree six and total derivation order five, since each application of $\Diamond$ introduces exactly one derivative.

Recall that, in a basis monomial of the free differential perm algebra, the last factor is distinguished, whereas all preceding factors may be permuted by the left-commutative identity. Consider a basis monomial occurring in the alternating sum \eqref{eq:standard-diamond}. If its distinguished last factor is $x_6^{(s)}$, then all five alternating variables $x_1,\ldots,x_5$ occur among the permutable factors. For the coefficient of such a monomial to be nonzero, the derivation orders of these five variables must be pairwise distinct; otherwise, transposing two variables having the same derivation order leaves the basis monomial unchanged and changes the sign of the summand. Hence their total derivation order is at least
\[
0+1+2+3+4=10,
\]
which is impossible because the total derivation order is five.

If the distinguished last factor is one of $x_1,\ldots,x_5$, then the remaining four alternating variables occur among the permutable factors. Again, their derivation orders must be pairwise distinct, and therefore their total derivation order is at least
\[
0+1+2+3=6>5.
\]
This is also impossible. Consequently every basis coefficient in $\operatorname{St}_6$ is zero, and hence
\[
\operatorname{St}_6(x_1,\ldots,x_6)=0.
\]

The same argument explains the commutative case: in a free commutative differential algebra all factors are permutable, and the four alternating variables in $\operatorname{St}_5$ would have to carry pairwise distinct derivation orders. Their total order would be at least $0+1+2+3=6$, whereas every monomial in $\operatorname{St}_5$ has total derivation order four.

On the other hand, $\operatorname{St}_5$ is nonzero in the free differential perm algebra. Indeed, in its expansion the basis monomial
\[
x_2x_5x_4'x_3''x_1'
\]
occurs with coefficient $-2$. Thus $\operatorname{St}_5$ is not an identity.

The computations of $\operatorname{St}_5$ and $\operatorname{St}_6$ are also included in the supplementary Mathematica code \cite{Code}.
\end{Proof}

\subsection{Polynomial algebra and the perm--Witt construction}
\label{subsec:perm-witt-construction-euler}

We begin with a convenient way to construct perm algebras from commutative associative algebras.

\begin{Proposition}\label{prop:perm-from-comm}
Let $A$ be a commutative associative $\mathbb K$-algebra generated by a set $X$, and let $\mathbb K X$ be the $\mathbb K$-vector space with basis $X$. We identify every $x\in X$ with the corresponding generator of $A$ and set
\[
P(A,X):=A\otimes_{\mathbb K}\mathbb K X.
\]
Define a bilinear multiplication on $P(A,X)$ by
\begin{equation}\label{eq:perm-tensor-product-polished}
(u\otimes x)\cdot(v\otimes y):=(uvx)\otimes y,
\qquad u,v\in A,\quad x,y\in X.
\end{equation}
Then $P(A,X)$ is a perm algebra.
\end{Proposition}

\begin{proof}
Let $u,v,w\in A$ and $x,y,z\in X$. Using the commutativity and associativity of $A$, we have
\[
\begin{aligned}
\bigl((u\otimes x)\cdot(v\otimes y)\bigr)\cdot(w\otimes z)
&=(uvwxy)\otimes z,\\
(u\otimes x)\cdot\bigl((v\otimes y)\cdot(w\otimes z)\bigr)
&=(uvwxy)\otimes z,
\end{aligned}
\]
so the product is associative. Moreover,
\[
\bigl((u\otimes x)\cdot(v\otimes y)\bigr)\cdot(w\otimes z)
=
\bigl((v\otimes y)\cdot(u\otimes x)\bigr)\cdot(w\otimes z),
\]
which is the left-commutative identity. Hence $P(A,X)$ is a perm algebra.
\end{proof}

Let $A=\mathbb K[X]$ be the polynomial algebra, where $X=\{x_1,\dots,x_n\}$. Set
\[
P_n:=A\otimes_{\mathbb K}\mathbb K X,
\]
and endow $P_n$ with the perm product
\begin{equation}\label{eq:perm-tensor-product-euler}
(u\otimes x_r)\cdot(v\otimes x_s):=(uvx_r)\otimes x_s,
\qquad u,v\in A,\quad r,s\in\{1,\dots,n\}.
\end{equation}

\medskip
\noindent\textbf{Derivations.}
For each $i\in\{1,\dots,n\}$ define a $\mathbb K$-linear map $D_i:P_n\to P_n$ by
\begin{equation}\label{eq:Di-euler}
\begin{aligned}
D_i(u\otimes x_j)
&:=\bigl(x_i\partial_i(u)\bigr)\otimes x_j
   +u\partial_i(x_j)\otimes x_i\\
&=\bigl(x_i\partial_i(u)\bigr)\otimes x_j
   +\delta_{ij}u\otimes x_i,
\end{aligned}
\end{equation}
where $\partial_i=\frac{\partial}{\partial x_i}$.

\begin{lemma}\label{lem:Di-derivation-commuting}
Each $D_i$ is a derivation of the associative product $\cdot$ on $P_n$, and the derivations commute:
\[
[D_i,D_j]=0
\qquad\text{for all }i,j\in\{1,\dots,n\}.
\]
\end{lemma}

\begin{proof}
It suffices to check the Leibniz rule on simple tensors
$a=u\otimes x_r$ and $b=v\otimes x_s$. By \eqref{eq:perm-tensor-product-euler},
\[
a\cdot b=(uvx_r)\otimes x_s.
\]
Applying \eqref{eq:Di-euler}, we obtain
\[
D_i(a\cdot b)=\bigl(x_i\partial_i(uvx_r)\bigr)\otimes x_s
+\delta_{is}(uvx_r)\otimes x_i.
\]
On the other hand,
\[
D_i(a)=(x_i\partial_i u)\otimes x_r+\delta_{ir}u\otimes x_i,
\qquad
D_i(b)=(x_i\partial_i v)\otimes x_s+\delta_{is}v\otimes x_i,
\]
and hence
\[
D_i(a)\cdot b
=\bigl((x_i\partial_i u)vx_r+\delta_{ir}uvx_i\bigr)\otimes x_s,
\]
while
\[
a\cdot D_i(b)
=\bigl(u(x_i\partial_i v)x_r\bigr)\otimes x_s
+\delta_{is}(uvx_r)\otimes x_i.
\]
Since
\[
x_i\partial_i(uvx_r)
=(x_i\partial_i u)vx_r+u(x_i\partial_i v)x_r+\delta_{ir}uvx_i,
\]
we obtain
\[
D_i(a\cdot b)=D_i(a)\cdot b+a\cdot D_i(b).
\]
Thus each $D_i$ is a derivation.

The commutativity $[D_i,D_j]=0$ follows from the commutativity of the Euler derivations $x_i\partial_i$ and $x_j\partial_j$ on $A$ together with formula \eqref{eq:Di-euler}. Equivalently, on the monomial basis
\[
x_1^{m_1}\cdots x_n^{m_n}\otimes x_r
\]
one has
\[
D_i\bigl(x_1^{m_1}\cdots x_n^{m_n}\otimes x_r\bigr)
=(m_i+\delta_{ir})
\bigl(x_1^{m_1}\cdots x_n^{m_n}\otimes x_r\bigr),
\]
so the operators $D_i$ commute pairwise.
\end{proof}

\medskip
\noindent\textbf{Perm--Witt algebra.}
Set
\[
W_{\mathrm{perm}}(n):=\bigoplus_{i=1}^n P_nD_i,
\]
where the summands $P_nD_i$ are regarded as labeled copies of $P_n$. We equip $W_{\mathrm{perm}}(n)$ with the bracket
\begin{equation}\label{eq:perm-witt-bracket-euler}
[aD_i,bD_j]:=a\cdot D_i(b)D_j-b\cdot D_j(a)D_i,
\qquad a,b\in P_n,
\end{equation}
for $i,j\in\{1,\dots,n\}$. Equivalently, introduce the bilinear product
\begin{equation}\label{eq:prec-euler}
(aD_i)\prec(bD_j):=\bigl(a\cdot D_i(b)\bigr)D_j,
\qquad\text{so that}\qquad
[X,Y]=X\prec Y-Y\prec X.
\end{equation}

\begin{proposition}\label{prop:prelie-perm-witt}
Let $(P,\cdot)$ be a perm algebra and let $D_1,\dots,D_n\in\Der(P)$ be commuting derivations. Set
\[
W_{\mathrm{perm}}:=\bigoplus_{i=1}^n PD_i
\]
as a direct sum of labeled copies of $P$. Then the product
\[
(aD_i)\prec(bD_j):=\bigl(a\cdot D_i(b)\bigr)D_j
\]
turns $W_{\mathrm{perm}}$ into a left-symmetric algebra.
\end{proposition}

\begin{proof}
Take $X=aD_i$, $Y=bD_j$, and $Z=cD_k$. The associator
\[
(X,Y,Z):=(X\prec Y)\prec Z-X\prec(Y\prec Z)
\]
is given by
\[
\begin{aligned}
(X,Y,Z)
&=\Bigl((aD_i(b))\cdot D_j(c)
-a\cdot D_i\bigl(b\cdot D_j(c)\bigr)\Bigr)D_k\\
&=-\bigl(a\cdot b\cdot D_iD_j(c)\bigr)D_k.
\end{aligned}
\]
Similarly,
\[
(Y,X,Z)=-\bigl(b\cdot a\cdot D_jD_i(c)\bigr)D_k.
\]
Since the derivations commute and the left-commutative identity gives
\[
a\cdot b\cdot t=b\cdot a\cdot t
\qquad\text{for all }a,b,t\in P,
\]
we obtain
\[
(X,Y,Z)=(Y,X,Z).
\]
Hence $W_{\mathrm{perm}}$ is left-symmetric.
\end{proof}

\begin{corollary}\label{prop:perm-witt-lie-euler}
The bracket \eqref{eq:perm-witt-bracket-euler} turns $W_{\mathrm{perm}}(n)$ into a Lie algebra.
\end{corollary}

\subsection{Explicit Lie multiplication tables for $W_{\mathrm{perm}}(1)$ and $W_{\mathrm{perm}}(2)$}

\noindent\textit{The case $n=1$.}
Let $A=\mathbb K[x]$, $X=\{x\}$, and $P_1=A\otimes_{\mathbb K}\mathbb Kx$. Put
\[
e_m:=x^m\otimes x,
\qquad
E_m:=e_mD_1,
\qquad m\ge0.
\]
Then
\[
e_m\cdot e_n=e_{m+n+1},
\qquad
D_1(e_n)=(n+1)e_n.
\]
Hence
\[
E_m\prec E_n=(n+1)E_{m+n+1},
\]
and therefore
\begin{equation}\label{eq:witt-one-variable}
[E_m,E_n]=(n-m)E_{m+n+1}.
\end{equation}
If $L_r:=E_{r-1}$ for $r\ge1$, then
\[
[L_r,L_s]=(s-r)L_{r+s},
\]
so this Lie algebra is isomorphic to the positive part of the classical Witt algebra.

\medskip
\noindent\textit{The case $n=2$.}
Let $A=\mathbb K[x,y]$, $X=\{x,y\}$, and $P_2=A\otimes_{\mathbb K}\mathbb KX$.
For $m,n\ge0$ set
\[
e_{m,n}^x:=x^my^n\otimes x,
\qquad
e_{m,n}^y:=x^my^n\otimes y,
\]
and define basis elements of $W_{\mathrm{perm}}(2)=P_2D_x\oplus P_2D_y$ by
\[
E_{m,n}^{\alpha,i}:=e_{m,n}^{\alpha}D_i,
\qquad
\alpha\in\{x,y\},\quad i\in\{x,y\}.
\]
The product in $P_2$ satisfies
\[
e_{m,n}^x\cdot e_{p,q}^{\beta}=e_{m+p+1,n+q}^{\beta},
\qquad
e_{m,n}^y\cdot e_{p,q}^{\beta}=e_{m+p,n+q+1}^{\beta},
\qquad \beta\in\{x,y\},
\]
and the derivations act by
\[
D_x(e_{m,n}^x)=(m+1)e_{m,n}^x,
\qquad
D_x(e_{m,n}^y)=m e_{m,n}^y,
\]
\[
D_y(e_{m,n}^x)=n e_{m,n}^x,
\qquad
D_y(e_{m,n}^y)=(n+1)e_{m,n}^y.
\]

\begin{corollary}\label{prop:perm-witt-lie-euler table}
The Lie algebra $(W_{\mathrm{perm}}(2),[-,-])$ has the following multiplication table.

\medskip
\noindent\emph{(i) $D_x$--$D_x$ block}
\begin{align*}
[ E_{m,n}^{x,x},E_{p,q}^{x,x}]&=(p-m)E_{m+p+1,n+q}^{x,x},\\
[ E_{m,n}^{x,x},E_{p,q}^{y,x}]&=pE_{m+p+1,n+q}^{y,x}-(m+1)E_{m+p,n+q+1}^{x,x},\\
[ E_{m,n}^{y,x},E_{p,q}^{x,x}]&=(p+1)E_{m+p,n+q+1}^{x,x}-mE_{m+p+1,n+q}^{y,x},\\
[ E_{m,n}^{y,x},E_{p,q}^{y,x}]&=(p-m)E_{m+p,n+q+1}^{y,x}.
\end{align*}

\smallskip
\noindent\emph{(ii) $D_y$--$D_y$ block}
\begin{align*}
[ E_{m,n}^{x,y},E_{p,q}^{x,y}]&=(q-n)E_{m+p+1,n+q}^{x,y},\\
[ E_{m,n}^{x,y},E_{p,q}^{y,y}]&=(q+1)E_{m+p+1,n+q}^{y,y}-nE_{m+p,n+q+1}^{x,y},\\
[ E_{m,n}^{y,y},E_{p,q}^{x,y}]&=qE_{m+p,n+q+1}^{x,y}-(n+1)E_{m+p+1,n+q}^{y,y},\\
[ E_{m,n}^{y,y},E_{p,q}^{y,y}]&=(q-n)E_{m+p,n+q+1}^{y,y}.
\end{align*}

\smallskip
\noindent\emph{(iii) mixed block $D_x$--$D_y$}
\begin{align*}
[ E_{m,n}^{x,x},E_{p,q}^{x,y}]&=(p+1)E_{m+p+1,n+q}^{x,y}-nE_{m+p+1,n+q}^{x,x},\\
[ E_{m,n}^{x,x},E_{p,q}^{y,y}]&=pE_{m+p+1,n+q}^{y,y}-nE_{m+p,n+q+1}^{x,x},\\
[ E_{m,n}^{y,x},E_{p,q}^{x,y}]&=(p+1)E_{m+p,n+q+1}^{x,y}-(n+1)E_{m+p+1,n+q}^{y,x},\\
[ E_{m,n}^{y,x},E_{p,q}^{y,y}]&=pE_{m+p,n+q+1}^{y,y}-(n+1)E_{m+p,n+q+1}^{y,x}.
\end{align*}
\end{corollary}

All remaining brackets follow from skew-symmetry. The same computations are verified in the supplementary
Mathematica code \cite{Code}.

\subsection{Explicit Leibniz multiplication table for $W_{\mathrm{perm}}(2)$}

Keeping the above notation and using Proposition~\ref{prop:leibniz-new-bracket}, we now modify the bracket so that it becomes Leibniz rather than Lie. This construction yields a family of infinite-dimensional Leibniz algebras.

On $W_{\mathrm{perm}}(n)$ consider the bilinear bracket
\begin{equation}\label{eq:leibniz-bracket-euler}
[aD_i,bD_j]_{\circ}:=D_j(a)\cdot bD_i-a\cdot D_i(b)D_j,
\qquad a,b\in P_n.
\end{equation}
By Proposition~\ref{prop:leibniz-new-bracket}, $(W_{\mathrm{perm}}(n),[-,-]_{\circ})$ is a Leibniz algebra.

\begin{corollary}\label{prop:wperm-1-2-structure-constants}
For all $m,n,p,q\ge0$, the Leibniz bracket on $W_{\mathrm{perm}}(2)$ is given by the following multiplication table.

\medskip
\noindent\emph{(a) Left factor $E_{m,n}^{x,*}$.}
\begin{align*}
[ E_{m,n}^{x,x},E_{p,q}^{x,x}]_{\circ}&=(m-p)E_{m+p+1,n+q}^{x,x},\\
[ E_{m,n}^{x,x},E_{p,q}^{x,y}]_{\circ}&=nE_{m+p+1,n+q}^{x,x}-(p+1)E_{m+p+1,n+q}^{x,y},\\
[ E_{m,n}^{x,y},E_{p,q}^{x,x}]_{\circ}&=(m+1)E_{m+p+1,n+q}^{x,y}-qE_{m+p+1,n+q}^{x,x},\\
[ E_{m,n}^{x,y},E_{p,q}^{x,y}]_{\circ}&=(n-q)E_{m+p+1,n+q}^{x,y},\\[1mm]
[ E_{m,n}^{x,x},E_{p,q}^{y,x}]_{\circ}&=(m+1-p)E_{m+p+1,n+q}^{y,x},\\
[ E_{m,n}^{x,x},E_{p,q}^{y,y}]_{\circ}&=nE_{m+p+1,n+q}^{y,x}-pE_{m+p+1,n+q}^{y,y},\\
[ E_{m,n}^{x,y},E_{p,q}^{y,x}]_{\circ}&=(m+1)E_{m+p+1,n+q}^{y,y}-(q+1)E_{m+p+1,n+q}^{y,x},\\
[ E_{m,n}^{x,y},E_{p,q}^{y,y}]_{\circ}&=(n-q-1)E_{m+p+1,n+q}^{y,y}.
\end{align*}

\medskip
\noindent\emph{(b) Left factor $E_{m,n}^{y,*}$.}
\begin{align*}
[ E_{m,n}^{y,x},E_{p,q}^{x,x}]_{\circ}&=(m-p-1)E_{m+p,n+q+1}^{x,x},\\
[ E_{m,n}^{y,x},E_{p,q}^{x,y}]_{\circ}&=(n+1)E_{m+p,n+q+1}^{x,x}-(p+1)E_{m+p,n+q+1}^{x,y},\\
[ E_{m,n}^{y,y},E_{p,q}^{x,x}]_{\circ}&=mE_{m+p,n+q+1}^{x,y}-qE_{m+p,n+q+1}^{x,x},\\
[ E_{m,n}^{y,y},E_{p,q}^{x,y}]_{\circ}&=(n+1-q)E_{m+p,n+q+1}^{x,y},\\[1mm]
[ E_{m,n}^{y,x},E_{p,q}^{y,x}]_{\circ}&=(m-p)E_{m+p,n+q+1}^{y,x},\\
[ E_{m,n}^{y,x},E_{p,q}^{y,y}]_{\circ}&=(n+1)E_{m+p,n+q+1}^{y,x}-pE_{m+p,n+q+1}^{y,y},\\
[ E_{m,n}^{y,y},E_{p,q}^{y,x}]_{\circ}&=mE_{m+p,n+q+1}^{y,y}-(q+1)E_{m+p,n+q+1}^{y,x},\\
[ E_{m,n}^{y,y},E_{p,q}^{y,y}]_{\circ}&=(n-q)E_{m+p,n+q+1}^{y,y}.
\end{align*}
\end{corollary}

The same computations are verified in the supplementary
Mathematica code \cite{Code}.

The above examples show that these constructions extend the classical Witt construction to the perm setting. In one variable the coefficient algebra $P_1$ is commutative, and \eqref{eq:witt-one-variable} recovers the positive part of the classical Witt algebra. In contrast, for $n\ge2$ the coefficient algebra $P_n$ is noncommutative, and this noncommutativity is reflected in the structure constants of $W_{\mathrm{perm}}(n)$. Thus the family $W_{\mathrm{perm}}(n)$ provides natural Witt-type Lie and Leibniz algebras associated with perm algebras.

\textbf{Declaration of interests.} The authors have no conflicts of interest to disclose.

\textbf{Acknowledgments.} This work was supported by the Science Committee of the Ministry of Education and Science of the Republic of Kazakhstan (Grant No. AP22683764).


\begin{thebibliography}{99}

\bibitem{BBGW}
C. Bai, R. Bai, L. Guo, Y. Wu, Transposed Poisson algebras, Novikov-Poisson algebras and 3-Lie algebras, Journal of Algebra, 632 (2023), 535--566.


\bibitem{CD-NLP}
J.\,M. Casas, T. Datuashvili,
Noncommutative Leibniz--Poisson algebras,
Comm. Algebra 34 (2006), no.\,7, 2507--2530.


\bibitem{Chapoton2001}
F. Chapoton,
Un endofoncteur de la cat{\'e}gorie des op{\'e}rades,
in: Dialgebras and Related Operads,
Lect. Notes Math., Vol.\,1763, Springer, Berlin, 2001, 105--110.


\bibitem{CuiHouOnline2026}
Z. Cui, B. Hou,
Averaging antisymmetric infinitesimal bialgebras and induced perm bialgebras,
J. Algebra Appl. (2027), Article 2750126.

\bibitem{DMS2027}
A. Dauletiyarova, A. Makhlouf, B. Sartayev, Varieties of initial dialgebras and some of their Koszul dual operads, Communications in Mathematics, 2027, 35(2).

\bibitem{DotsenkoIsmailovUmirbaev2023}
V. Dotsenko, N. Ismailov, U. Umirbaev,
Polynomial identities in Novikov algebras,
Math. Z. 303 (2023), no.\,3.

\bibitem{Dzhumadil2002}
A.\,S. Dzhumadil'daev,
Novikov--Jordan algebras,
Comm. Algebra 30 (2002), 5207--5240.

\bibitem{Dzhumadil2005}
A.\,S. Dzhumadil'daev,
Special identity for Novikov--Jordan algebras,
Comm. Algebra 33 (2005), no.\,5, 1279--1287.


\bibitem{DzhumadilIsmailov2026Null}
A.\,S. Dzhumadil'daev, N. Ismailov,
Null Lagrangians in free Novikov algebras,
arXiv:2601.11168, 2026.


\bibitem{GK2014}
V.\,Yu. Gubarev, P.\,S. Kolesnikov,
Operads of decorated trees and their duals,
Commentat. Math. Univ. Carolin. 55 (2014), no.\,4, 421--445.

\bibitem{Hou2024}
B. Hou,
Extending structures for perm algebras and perm bialgebras,
J. Algebra 649 (2024), 392--432.


\bibitem{KaygorodovMashurov2024}
I. Kaygorodov, F. Mashurov,
Mutations of perm algebras,
Rev. R. Acad. Cienc. Exactas F{\'\i}s. Nat. Ser. A Mat. 118 (2024), Article 166.


\bibitem{KS2026}
P. S. Kolesnikov, B. K. Sartayev, On the Dong Property for a binary quadratic operad, Journal of Algebra 691 (2026), 428--452.

\bibitem{KolesnikovSartayev2022}
P.\,S. Kolesnikov, B.\,K. Sartayev,
On the embedding of left-symmetric algebras into differential Perm-algebras,
Comm. Algebra 50 (2022), no.\,8, 3246--3260.


\bibitem{KunanbayevSartayev2026}
A. Kunanbayev, B.\,K. Sartayev,
Binary perm algebras and alternative algebras,
Comm. Algebra 54 (2026), no.\,1, 299--307.




\bibitem{LinQTPreprint2025}
Y. Lin,
Quasi-triangular and factorizable perm bialgebras,
J. Algebra Appl. (2027), Article 2750184.

\bibitem{LinZhouBai2025}
Y. Lin, P. Zhou, C. Bai,
Infinite-dimensional Lie bialgebras via affinization of perm bialgebras and pre-Lie bialgebras,
J. Algebra 663 (2025), 210--258.



\bibitem{Loday2001}
J.-L. Loday,
Dialgebras,
in: Dialgebras and Related Operads,
Lect. Notes Math., Vol.\,1763, Springer, Berlin, 2001, 7--66.

\bibitem{Code}
F.~A. Mashurov and B.~K. Sartayev,
Mathematica code for ``Identities in Differential Perm Algebras'',
GitHub repository,
\url{https://github.com/farukhmashurov/differential-perm-algebras-code},
(2026).


\bibitem{MashurovSartayev2024}
F.\,A. Mashurov, B.\,K. Sartayev,
Metabelian Lie and perm algebras,
J. Algebra Appl. 23 (2024), no.\,4, 2450065.



\bibitem{PeiBaiGuoNi2020}
J. Pei, C. Bai, L. Guo, X. Ni,
Replicators, Manin white product of binary operads and average operators,
in: New Trends in Algebra and Combinatorics,
World Scientific, 2020, 317--353.


\bibitem{SartayevDzhumadildaevMashurov2025}
B. Sartayev, A. Dzhumadil'daev, F. Mashurov,
Novikov Dialgebras and Perm Algebras,
Bull. Iran. Math. Soc. 51 (2025), Article 51.


\bibitem{Vallette2008}
B. Vallette,
Manin products, Koszul duality, Loday algebras and Deligne conjecture,
J. Reine Angew. Math. 620 (2008), 105--164.



\bibitem{Xu1994}
P. Xu, Noncommutative Poisson Algebras, American Journal of Mathematics, 116, 1 (1994), 101-125.

\bibitem{ZhaoLiuPreprint2025}
Q. Zhao, G. Liu,
A further study on averaging commutative and cocommutative infinitesimal bialgebras
and special apre-perm bialgebras,
arXiv:2510.09208 (2025).


\end{thebibliography}
\end{document}